\documentclass[11pt]{amsart}
\usepackage[top=1in, bottom=1in, left=1in, right=1in]{geometry}
\usepackage{times}
\usepackage{amssymb,amsmath,amsthm}
\usepackage[dvipsnames]{xcolor}
\usepackage[colorlinks=true, linkcolor=Cerulean, citecolor=ForestGreen, urlcolor=Plum]{hyperref}   
\numberwithin{equation}{section}

\theoremstyle{plain}
\newtheorem{thm}{Theorem}
\newtheorem{lem}{Lemma}[section]
\newtheorem{prop}[lem]{Proposition}
\newtheorem{cor}[thm]{Corollary}

\newcommand{\Q}{\mathbb{Q}}
\newcommand{\R}{\mathbb{R}}
\newcommand{\calZ}{\mathcal{Z}}
\renewcommand{\epsilon}{\varepsilon}
\renewcommand{\bar}[1]{\overline{#1}}
\renewcommand{\tilde}{\widetilde}
\renewcommand{\le}{\leqslant}

\renewcommand{\ge}{\geqslant}

\DeclareMathOperator{\Gal}{Gal}
\DeclareMathOperator{\Fix}{Fix}
\DeclareMathOperator{\Frob}{Frob}

\begin{document}

\title[Lower bounds for Dedekind zeta moments]
{Sharp lower bounds for shifted moments of Dedekind zeta functions}

\author{Benjamin Durkan}
\address{\raggedright Benjamin Durkan \\
Department of Mathematics\protect\linebreak
The University of Manchester\\
Oxford Road, Manchester, M13 9PL, United Kingdom}
\email{benjamin.durkan@manchester.ac.uk}

\author{Nilmoni Karak}
\address{\raggedright Nilmoni Karak\\
Department of Mathematics\protect\linebreak
Indian Institute of Technology\\
Kharagpur\\
Kharagpur 721302, India}
\email{nilmonikarak@gmail.com}
\email{nilmonimath@kgpian.iitkgp.ac.in}

\author{Kamalakshya Mahatab}
\address{\raggedright Kamalakshya Mahatab\\
Department of Mathematics\protect\linebreak
Indian Institute of Technology\\
Kharagpur\\
Kharagpur 721302, India}
\email{accessing.infinity@gmail.com}
\email{kamalakshya@maths.iitkgp.ac.in}

\date{}

\begin{abstract}
Let $K_1,\cdots,K_r$ be fixed number fields, and let $L$ be the
compositum of their Galois closures.  Assuming GRH for $\zeta_L$, we prove
a sharp lower bound for products of shifted Dedekind zeta functions on the
critical line, for arbitrary fixed positive real exponents and uniformly
for shifts of size at most $T/2$.  The correlation factor is expressed as
a product of Dedekind zeta functions of the fixed fields of double-coset
stabilisers in $\textrm{Gal}(L/\mathbb{Q})$.  Combined with the corresponding upper
bound by the authors, determines the order of magnitude
of these shifted moments for both Galois and non-Galois fields.
\end{abstract}

\maketitle

\section{Introduction}\label{sec:introduction}

Shifted moments of $L$-functions measure correlations between values at
different ordinates of the critical line.  When the shifts coalesce they
behave like higher ordinary moments, while separated shifts exhibit the
pairwise Euler correlations predicted by Conrey, Farmer, Keating,
Rubinstein, and Snaith \cite{CFKRS}.

For Dedekind zeta functions, the situation is complicated by
non-primitivity.  If $K$ has Galois closure $L$, with
$$
G=\Gal(L/\Q),\qquad H=\Gal(L/K),
$$
then the expected exponent in the $2k$th moment is
$$
\rho_K k^2,\qquad \rho_K=|H\backslash G/H|.
$$
Thus $\rho_K=[K:\Q]$ when $K/\Q$ is Galois, whereas $\rho_K=2$ for a
non-Galois cubic field with Galois closure group $S_3$.  This is consistent
with Heap's conjectures for moments of non-primitive $L$-functions
\cite{HeapDedekind}.

Sharp moment bounds for the Riemann zeta function are now known in a broad
range of shifted and unshifted settings
\cite{Soundararajan,Harper,Chandee,NgShenWong,CurranCorrelation,
RadziwillSoundararajan,HeapSoundararajan,CurranLower}.  For Dedekind zeta
functions, previous work includes special asymptotics, conditional results
for Galois fields, and unconditional cumulative unshifted lower bounds
\cite{Motohashi,MilinovichTurnageButterbaugh,Hagen,AkbaryFodden,Sono,
Krishnamoorthy}.  The purpose of this paper is to obtain the sharp dyadic
lower bound for several Dedekind zeta functions simultaneously, allowing
arbitrary fixed positive real exponents and shifts as large as $T/2$.
Combined with the corresponding upper bound of
\cite{DurkanKarakMahatab}, this determines the order of magnitude of the
shifted moments.

Let $K_1,\ldots,K_r$ be fixed number fields.  Fix embeddings into an
algebraic closure of $\Q$, let $L$ be the compositum of their Galois
closures, and let
$$
G=\Gal(L/\Q),\qquad H_i=\Gal(L/K_i).
$$
Let $\chi_i$ be the permutation character of $G$ acting on $G/H_i$.  For
$1\le i,j\le r$, define
\begin{equation}\label{eq:double-coset-zeta}
\calZ_{ij}(s)=
\prod_{H_i gH_j\in H_i\backslash G/H_j}
\zeta_{L^{H_i\cap gH_jg^{-1}}}(s).
\end{equation}
Conjugate representatives give isomorphic fixed fields, so the product is
well-defined.  For $T\ge3$, real shifts
$\mathbf b=(b_1,\ldots,b_r)$, and fixed exponents
$\mathbf a=(a_1,\ldots,a_r)$ with $a_i>0$, set
\begin{equation}\label{eq:B-definition}
B(T,\mathbf b,\mathbf a)=
\prod_{i,j=1}^r
\left|
\calZ_{ij}\left(1+\frac{1}{\log T}+i(b_i-b_j)\right)
\right|^{a_i a_j/4}.
\end{equation}

Our main result is the following.

\begin{thm}\label{thm:main}
Let $K_1,\ldots,K_r$ be fixed number fields, embedded as above, let $L$ be
the compositum of their Galois closures, and let $a_1,\ldots,a_r>0$ be
fixed.  Assume GRH for $\zeta_L(s)$.  There is
$$
T_0=T_0(K_1,\ldots,K_r,\mathbf a)\ge3
$$
such that, for $T\ge T_0$, uniformly for real shifts satisfying
$|b_i|\le T/2$,
\begin{equation}\label{eq:main-lower-bound}
\int_T^{2T}\prod_{j=1}^r
\left|\zeta_{K_j}\left(\tfrac12+i(t+b_j)\right)\right|^{a_j}\,dt
\gg_{K_1,\ldots,K_r,\mathbf a}
T B(T,\mathbf b,\mathbf a).
\end{equation}
\end{thm}

The upper bound proved in \cite[Theorem~1.1]{DurkanKarakMahatab} is of the
same order under the same hypotheses.  Hence
\begin{equation}\label{eq:two-sided-order}
\int_T^{2T}\prod_{j=1}^r
\left|\zeta_{K_j}\left(\tfrac12+i(t+b_j)\right)\right|^{a_j}\,dt
\asymp_{K_1,\ldots,K_r,\mathbf a}
T B(T,\mathbf b,\mathbf a),
\end{equation}
uniformly for $|b_i|\le T/2$.  When every $K_i=\Q$, this recovers the
positive-exponent case of Curran's shifted lower bound
\cite[Theorem~1.1]{CurranLower}.

The theorem has several immediate unshifted consequences.  Write
$$
\langle\varphi,\psi\rangle_G
=\frac1{|G|}\sum_{\sigma\in G}
\varphi(\sigma)\overline{\psi(\sigma)}.
$$

\begin{cor}\label{cor:unshifted-mixed}
Under the hypotheses of Theorem~\ref{thm:main}, for all sufficiently large
$T$,
\begin{equation}\label{eq:unshifted-mixed}
\int_T^{2T}\prod_{i=1}^r
\left|\zeta_{K_i}\left(\tfrac12+it\right)\right|^{a_i}\,dt
\gg_{K_1,\ldots,K_r,\mathbf a}
T(\log T)^{\frac14
\left\langle\sum_i a_i\chi_i,\sum_i a_i\chi_i\right\rangle_G}.
\end{equation}
Equivalently, the exponent is
$$
\frac14\sum_{i,j=1}^r
a_i a_j |H_i\backslash G/H_j|.
$$
\end{cor}

\begin{cor}\label{cor:single-field}
Let $K$ be a fixed number field, let $L$ be its Galois closure, and set
$$
G=\Gal(L/\Q),\qquad H=\Gal(L/K),\qquad
\rho_K=|H\backslash G/H|.
$$
Assume GRH for $\zeta_L$.  Then, for every fixed real $k>0$ and all
sufficiently large $T$,
\begin{equation}\label{eq:single-field}
\int_T^{2T}
\left|\zeta_K\left(\tfrac12+it\right)\right|^{2k}\,dt
\gg_{K,k}
T(\log T)^{\rho_Kk^2}.
\end{equation}
If $K$ is a non-Galois cubic field whose Galois closure has group $S_3$,
then $\rho_K=2$.
\end{cor}

Together with \cite[Theorem~1.1]{DurkanKarakMahatab}, this gives
$$
\int_T^{2T}
\left|\zeta_K\left(\tfrac12+it\right)\right|^{2k}\,dt
\asymp_{K,k}
T(\log T)^{\rho_Kk^2}.
$$

\begin{cor}\label{cor:galois}
Let $K/\Q$ be a fixed Galois extension, and assume GRH for $\zeta_K$.
Then, for every fixed real $k>0$ and all sufficiently large $T$,
\begin{equation}\label{eq:galois}
\int_T^{2T}
\left|\zeta_K\left(\tfrac12+it\right)\right|^{2k}\,dt
\gg_{K,k}
T(\log T)^{[K:\Q]k^2}.
\end{equation}
\end{cor}

The moment estimates also yield a joint large-deviation consequence.

\begin{cor}
Let $K_1,\ldots,K_r$ be fixed number fields, and let $L$ be the compositum
of their Galois closures over $\Q$.  Let
$$
G=\Gal(L/\Q),\qquad
H_j=\Gal(L/K_j),\qquad
\rho_{K_j}=|H_j\backslash G/H_j|.
$$
Assume GRH for $\zeta_L$, and let
$V_j\asymp\sqrt{\log\log T}$.  Then
\begin{align*}
\frac1T\operatorname{meas}\Bigg\{
t\in[T,2T]:
\log\left|\zeta_{K_j}\left(\frac12+it\right)\right|
\ge
V_j\sqrt{\frac{\rho_{K_j}}2\log\log T},
\ 1\le j\le r
\Bigg\}
\\
\asymp_{K_1,\ldots,K_r}
\exp\left(
-(1+o(1))
\frac{V_1^2+\cdots+V_r^2}{2}
\right).
\end{align*}
\end{cor}

\subsection{Overview of the proof}

Set $q_i=a_i/2$ and, in $\Re s>1$, consider
\[
F_{\mathbf b}(s)
=
\prod_{i=1}^r\zeta_{K_i}(s+ib_i)^{q_i}
=
\sum_{n\ge1}\frac{c_{\mathbf b}(n)}{n^s}.
\]
At an unramified prime,
\[
c_{\mathbf b}(p)
=
\frac12\sum_{i=1}^r
a_i\chi_i(\Frob_p)p^{-ib_i}.
\]
The square $|c_{\mathbf b}(p)|^2$ therefore involves products
$\chi_i\chi_j$.  The diagonal $G$-action on
$G/H_i\times G/H_j$ decomposes into orbits indexed by
$H_i\backslash G/H_j$, producing precisely the fixed-field factors
$\calZ_{ij}$ in \eqref{eq:B-definition}.  This identifies the
square-coefficient Dirichlet series with $B(T,\mathbf b,\mathbf a)$ up to
bounded polynomial losses.

A truncation at length $T^{1/2}$ then gives the required weighted
mean-square lower bound slightly to the right of the critical line.  To
handle fractional powers and the shifted poles, we regularise by
\[
\Phi_K(z)=\frac{z-1}{z+1}\zeta_K(z).
\]
Under GRH for $\zeta_L$, the Aramata--Brauer theorem implies that these
functions are analytic and zero-free in $\Re z>1/2$.  Gaussian-weighted
Gabriel convexity is then applied twice: first to control the truncation
error near the critical line, and then to transfer the resulting lower
bound to $\Re s=1/2$.  The Gaussian weight localises the integral to
$[T,2T]$.

\subsection*{Notation}

All implied constants may depend on the fixed fields and on $\mathbf a$,
but not on $T$ or $\mathbf b$.  The finite set of rational primes ramified
in $L$ is denoted by $S_L$.

\subsection{Overview of the proof}
For $q_i=a_i/2$, consider in $\Re s>1$ the Euler-product branch
\begin{equation*}
F_{\mathbf b}(s)=
\prod_{i=1}^r\zeta_{K_i}(s+ib_i)^{q_i}
=\sum_{n\ge1}\frac{c_{\mathbf b}(n)}{n^s}.
\end{equation*}
The first coefficient at an unramified prime is
\begin{equation*}
c_{\mathbf b}(p)=
\frac12\sum_{i=1}^r
a_i\chi_i(\Frob_p)p^{-ib_i}.
\end{equation*}
Proposition~\ref{prop:square-series-B} shows that the Euler product for
$\sum_n|c_{\mathbf b}(n)|^2n^{-1-\lambda/\log T}$ is comparable to
$B(T,\mathbf b,\mathbf a)$, with a loss bounded by a fixed power of
$2+\lambda$.  Rankin's argument then gives a
Dirichlet polynomial of length $T^{1/2}$ whose weighted mean square at
$\Re s=1/2+\eta/\log T$ is
$\gg \eta^{-C}TB(T,\mathbf b,\mathbf a)$.

The fractional-power product $F_{\mathbf b}$ cannot be continued directly
through the shifted poles.  We replace each zeta function by
\begin{equation*}
\Phi_K(z)=\frac{z-1}{z+1}\zeta_K(z).
\end{equation*}
GRH and the Aramata--Brauer theorem show that $\Phi_{K_i}$ is analytic and
zero-free in $\Re z>1/2$.  Hence the corresponding fractional powers are
analytic there.  On $\Re s=5/4$, their difference from the length
$T^{1/2}$ polynomial has weighted mean square at most $T^{7/8}$.  Gabriel
convexity moves this estimate to $1/2+\eta/\log T$, with a factor
$\exp(-c\eta)$.  A fixed large $\eta$ makes the error smaller than the
polynomial mean square.  A second application of Gabriel convexity moves
the resulting lower bound to the critical line.  The Gaussian weight is
supported, up to an exponentially small tail, inside $[T,2T]$.

\subsection*{Notation}
All implied constants may depend on the fixed fields and on
$\mathbf a$, but not on $T$ or $\mathbf b$.  The finite set of
rational primes ramified in $L$ is denoted by $S_L$.

\section{Euler coefficients and double cosets}\label{sec:euler-coefficients}

We first record the analytic consequence of the standing hypothesis.

\begin{lem}\label{lem:regularised-zeta}
Assume GRH for $\zeta_L$.  For every intermediate field $K$ of $L/\Q$,
the function
\begin{equation}\label{eq:Phi-definition}
\Phi_K(z)=\frac{z-1}{z+1}\zeta_K(z)
\end{equation}
is analytic and zero-free in $\Re z>1/2$.
\end{lem}

\begin{proof}
The apparent singularity at $z=1$ is removable, and its value there is
half the positive residue of $\zeta_K$ at $1$.  The Aramata--Brauer theorem
states that if $E/F$ is a finite Galois extension, then
$\zeta_E(s)/\zeta_F(s)$ is entire \cite{Aramata,Brauer}.  Here $L/K$ is
Galois, so $\zeta_L(s)/\zeta_K(s)$ is entire.  A zero of
$\zeta_K$ in $\Re z>1/2$ would therefore be a zero of $\zeta_L$ there.
GRH excludes such a zero.  The factor $(z-1)/(z+1)$ has no remaining zero
or pole in this half-plane.  Thus $\Phi_K$ is analytic and zero-free there.
Since the half-plane is simply connected, $\Phi_K$ also admits an analytic
logarithm there.  This proves the lemma.
\end{proof}

Let $q_i=a_i/2$.  In $\Re s>1$, define
\begin{equation}\label{eq:F-dirichlet-series}
F_{\mathbf b}(s)
=\prod_{i=1}^r\zeta_{K_i}(s+ib_i)^{q_i}
=\sum_{n=1}^{\infty}\frac{c_{\mathbf b}(n)}{n^s},
\end{equation}
where the powers are defined using the logarithms furnished by the
absolutely convergent Euler products.  The coefficients
$c_{\mathbf b}(n)$ are multiplicative in $n$.

\begin{lem}\label{lem:coefficient-bounds}
For every $\epsilon>0$,
\begin{equation*}
|c_{\mathbf b}(n)|\ll_{\epsilon}n^{\epsilon},
\end{equation*}
uniformly for real $\mathbf b$.  If $p\notin S_L$, then
\begin{equation}\label{eq:first-prime-coefficient}
c_{\mathbf b}(p)=
\sum_{i=1}^r q_i\chi_i(\Frob_p)p^{-ib_i}.
\end{equation}
Moreover, for $\sigma>1$,
\begin{equation}\label{eq:square-euler-series}
\log\sum_{n=1}^{\infty}
\frac{|c_{\mathbf b}(n)|^2}{n^\sigma}
=\sum_p\frac{|c_{\mathbf b}(p)|^2}{p^\sigma}+O(1),
\end{equation}
where the error is uniform for $\sigma>1$ and for real
$\mathbf b$.
\end{lem}

\begin{proof}
At a rational prime $p$, every Euler factor of a shifted Dedekind zeta
function has the form
\begin{equation*}
\prod_{\mathfrak p\mid p}
(1-p^{-f_{\mathfrak p}(s+ib_i)})^{-q_i},
\end{equation*}
where $f_{\mathfrak p}$ denotes the residue degree of $\mathfrak p$ over
$p$, so that $N\mathfrak p=p^{f_{\mathfrak p}}$.  Since
$|p^{-ib_i}|=1$, its coefficient of $p^{-ms}$ is bounded by
$O((m+1)^A)$, uniformly in $\mathbf b$, for a fixed $A$.  The same bound
holds after the finitely many fields are multiplied.  The usual divisor
estimate gives the first
assertion.  The coefficient of $p^{-s}$ is obtained only from residue
degree one primes.  Equation \eqref{eq:first-prime-coefficient} follows from 
$$\frac{\Lambda_{K_i}}{\log p}=\chi_i(\Frob_p).$$

The Euler factor of the series on the left of
\eqref{eq:square-euler-series} is
\begin{equation*}
1+\frac{|c_{\mathbf b}(p)|^2}{p^\sigma}
+O\left(\frac{1}{p^{2\sigma}}
\sum_{m\ge2}\frac{(m+1)^{2A}}{p^{m-2}}
\right).
\end{equation*}
The sum in parentheses is bounded uniformly in $p$.  For all sufficiently
large $p$, the logarithm may therefore be expanded to first order with an
$O(p^{-2\sigma})$ remainder.  The finitely many smaller Euler factors have
logarithms bounded uniformly for $\sigma>1$ and real $\mathbf b$, because
their coefficients satisfy the same polynomial bound in $m$.  Summing over
the primes proves \eqref{eq:square-euler-series}.
\end{proof}

The next proposition is the Frobenius comparison needed for the lower-bound
argument; it is parallel to \cite[Section~2.3]{DurkanKarakMahatab}.

\begin{prop}\label{prop:square-series-B}
There is $A_0\ge1$ such that, uniformly for
\begin{equation*}
1\le\lambda\le\tfrac12\log T
\qquad\hbox{and}\qquad |b_i|\le T/2,
\end{equation*}
one has
\begin{equation}\label{eq:H-lambda-comparison}
(2+\lambda)^{-A_0}B(T,\mathbf b,\mathbf a)
\ll H_\lambda(T)
\ll(2+\lambda)^{A_0}B(T,\mathbf b,\mathbf a),
\end{equation}
where
\begin{equation}\label{eq:H-lambda-definition}
H_\lambda(T)=\sum_{n=1}^{\infty}
\frac{|c_{\mathbf b}(n)|^2}{n^{1+\lambda/\log T}}.
\end{equation}
Also, for some $A_1\ge1$,
\begin{equation}\label{eq:B-poly}
(\log T)^{-A_1}\ll B(T,\mathbf b,\mathbf a)
\ll(\log T)^{A_1}.
\end{equation}
\end{prop}

\begin{proof}
For $p\notin S_L$, equation~\eqref{eq:first-prime-coefficient} gives
\begin{equation}\label{eq:prime-square-expansion}
|c_{\mathbf b}(p)|^2
=\frac14\sum_{i,j=1}^r
a_i a_j\chi_i(\Frob_p)\chi_j(\Frob_p)
p^{-i(b_i-b_j)}.
\end{equation}
The diagonal action of $G$ on $G/H_i\times G/H_j$ has character
$\chi_i\chi_j$.  Its orbits are indexed by $H_i\backslash G/H_j$, and
the stabiliser of $(H_i,gH_j)$ is $H_i\cap gH_jg^{-1}$.  Decomposing this
$G$-set into its orbits gives, for every $\sigma\in G$,
\begin{equation}\label{eq:mackey-character-identity}
\begin{split}
\chi_i(\sigma)\chi_j(\sigma)
&=\sum_{H_i gH_j\in H_i\backslash G/H_j}
\Bigl|\Fix\Bigl(\sigma:G/(H_i\cap gH_jg^{-1})\\
&\hspace{9em}\longrightarrow
G/(H_i\cap gH_jg^{-1})\Bigr)\Bigr|.
\end{split}
\end{equation}
For $\Re s>1$, take the logarithms defined by the absolutely convergent
Euler products.  At an unramified prime $p$, the coefficient of $p^{-s}$
in the logarithm of the factor attached to
$L^{H_i\cap gH_jg^{-1}}$ is the number of fixed points of $\Frob_p$ on
$G/(H_i\cap gH_jg^{-1})$.  It follows from
\eqref{eq:mackey-character-identity} that
\begin{equation}\label{eq:double-coset-log-euler}
\log\calZ_{ij}(s)
=\sum_{p\notin S_L}
\frac{\chi_i(\Frob_p)\chi_j(\Frob_p)}{p^s}+O(1),
\end{equation}
where the $O(1)$ term is uniform for $\Re s>1$.  Indeed, the ramified
primes form a fixed
finite set, and the terms of prime-power exponent at least two are bounded
in absolute value by a fixed multiple of
$\sum_p\sum_{m\ge2}p^{-m\Re s}$.

Set $s=1+1/\log x+iu$.  Since the character values are bounded, Mertens'
estimate and partial summation give
\begin{equation*}
\sum_{p\le x}\frac{1-p^{-1/\log x}}p
+\sum_{p>x}\frac1{p^{1+1/\log x}}\ll1.
\end{equation*}
Comparing \eqref{eq:double-coset-log-euler} with the truncated prime sum
and taking real parts therefore gives, uniformly for real $u$ and $x\ge3$,
\begin{equation}\label{eq:double-coset-prime-sum}
\sum_{\substack{p\le x\\p\notin S_L}}
\frac{\chi_i(\Frob_p)\chi_j(\Frob_p)
\cos(u\log p)}{p}
=\log\left|\calZ_{ij}
\left(1+\frac{1}{\log x}+iu\right)\right|+O(1).
\end{equation}
The terms indexed by $(i,j)$ and $(j,i)$ in
\eqref{eq:prime-square-expansion} are complex conjugates, so their
imaginary parts cancel when the ordered pairs are summed.
Equations~\eqref{eq:prime-square-expansion} and
\eqref{eq:double-coset-prime-sum}, summed over the ordered pairs $(i,j)$,
give
\begin{equation}\label{eq:prime-sum-B}
\sum_{p\le T}\frac{|c_{\mathbf b}(p)|^2}{p}
=\log B(T,\mathbf b,\mathbf a)+O(1).
\end{equation}

Let $\delta=\lambda/\log T$.  Since
$|c_{\mathbf b}(p)|\ll1$, Mertens' estimate and partial summation give
\begin{align}
&\sum_{p\le T}\frac{1-p^{-\delta}}p
+\sum_{p>T}\frac1{p^{1+\delta}}
\ll \log(2+\lambda).                                    \label{eq:smoothing-cost}
\end{align}
For example, the first sum is bounded by a constant multiple of
\begin{equation*}
1+\int_{1/\log T}^{1}\frac{1-e^{-\lambda v}}v\,dv,
\end{equation*}
and the second by a constant multiple of
$1+\int_1^\infty e^{-\lambda v}v^{-1}\,dv$.
Combining \eqref{eq:square-euler-series}, \eqref{eq:prime-sum-B}, and
\eqref{eq:smoothing-cost} yields
\begin{equation*}
\log H_\lambda(T)
=\log B(T,\mathbf b,\mathbf a)
+O(\log(2+\lambda)).
\end{equation*}
This proves \eqref{eq:H-lambda-comparison}.

Every Dedekind zeta function in \eqref{eq:B-definition}, evaluated at
$1+1/\log T+iu$, has modulus at most its value at
$1+1/\log T$, which is $O(\log T)$.  If $F$ is one of the fixed fields
occurring in that definition and $\sigma=1+1/\log T$, then its Euler
product also gives
\begin{equation*}
|\zeta_F(\sigma+iu)|^{-1}
\le \prod_{\mathfrak p}\left(1+(N\mathfrak p)^{-\sigma}\right)
=\frac{\zeta_F(\sigma)}{\zeta_F(2\sigma)}
\ll\log T.
\end{equation*}
There are only finitely many factors.  This proves \eqref{eq:B-poly}.
\end{proof}

\section{A short polynomial}\label{sec:short-polynomial}

The parameter $\eta$ measures the displacement from the critical line.  The
next lemma records the required polynomial dependence on $\eta$.

\begin{lem}\label{lem:coefficient-truncation}
There exist constants $A_2\ge1$ and $\eta_0\ge1$, depending only on
the fixed fields and on $\mathbf a$, such that the following
holds.  Let $\eta\ge\eta_0$ be fixed, and suppose that $T$ is sufficiently
large in terms of $\eta$.  Set
\begin{equation}\label{eq:N-gamma}
N=T^{1/2},\qquad \gamma=\frac12+\frac{\eta}{\log T}.
\end{equation}
Then, uniformly for $|b_i|\le T/2$,
\begin{align}
\sum_{n\le N}\frac{|c_{\mathbf b}(n)|^2}{n^{2\gamma}}
&\gg \eta^{-A_2}B(T,\mathbf b,\mathbf a), \label{eq:truncated-lower}\\
\sum_{n\le N}\frac{|c_{\mathbf b}(n)|^2}{n}
&\ll B(T,\mathbf b,\mathbf a).                 \label{eq:truncated-upper}
\end{align}
\end{lem}

\begin{proof}
For $T$ sufficiently large in terms of $\eta$, both $\eta$ and $2\eta$
lie in the range of Proposition~\ref{prop:square-series-B}.  Hence
\begin{equation*}
H_{2\eta}(T)\gg \eta^{-A_0}B(T,\mathbf b,\mathbf a).
\end{equation*}
For $n>N$, Rankin's inequality gives
\begin{equation*}
n^{-1-2\eta/\log T}
\le N^{-\eta/\log T}n^{-1-\eta/\log T}
=e^{-\eta/2}n^{-1-\eta/\log T}.
\end{equation*}
Consequently,
\begin{equation*}
\sum_{n>N}\frac{|c_{\mathbf b}(n)|^2}
{n^{1+2\eta/\log T}}
\ll e^{-\eta/2}(2+\eta)^{A_0}
B(T,\mathbf b,\mathbf a).
\end{equation*}
Choose $\eta_0$ so that this is at most half the preceding lower bound
whenever $\eta\ge\eta_0$.  This proves \eqref{eq:truncated-lower}, after
increasing $A_2$.

If $n\le N$, then $n^{1/\log T}\le e^{1/2}$.  Hence
\begin{equation*}
\sum_{n\le N}\frac{|c_{\mathbf b}(n)|^2}{n}
\le e^{1/2}H_1(T)\ll B(T,\mathbf b,\mathbf a),
\end{equation*}
which proves \eqref{eq:truncated-upper}.
\end{proof}

Define
\begin{equation}\label{eq:S-definition}
S_N(s)=\sum_{n\le N}\frac{c_{\mathbf b}(n)}{n^s}.
\end{equation}
Let
\begin{equation}\label{eq:gaussian-weight}
J_T=[5T/4,7T/4],\qquad
w_T(t)=\int_{J_T}e^{-(t-\tau)^2}\,d\tau,
\end{equation}
and, for a function $U$ on a vertical line, write
\begin{equation}\label{eq:weighted-norm}
\mathcal M_\sigma(U)=
\int_{-\infty}^{\infty}|U(\sigma+it)|^2w_T(t)\,dt.
\end{equation}

\begin{lem}\label{lem:weighted-mean-value}
For $N\ge1$ and arbitrary complex numbers $d_n$,
\begin{equation}\label{eq:weighted-mean-value}
\int_{-\infty}^{\infty}
\left|\sum_{n\le N}d_n n^{-it}\right|^2w_T(t)\,dt
=\frac{\sqrt\pi T}{2}\sum_{n\le N}|d_n|^2
+O\left(N\log(2N)\sum_{n\le N}|d_n|^2\right).
\end{equation}
The absolute constant in the error does not depend on the $d_n$.
\end{lem}

\begin{proof}
Expanding the square and applying Fubini's theorem, we obtain
\begin{align*}
&\int_{\R}\left|\sum_{n\le N}d_n n^{-it}\right|^2w_T(t)\,dt\\
&\quad=\sqrt\pi\sum_{m,n\le N}d_m\bar{d_n}
\exp\left(-\frac14\log^2\frac mn\right)
\int_{J_T}e^{-i\tau\log(m/n)}\,d\tau.
\end{align*}
When $m=n$, the inner integral is $T/2$, and hence the diagonal
contribution gives the main term in \eqref{eq:weighted-mean-value}.  If
$m\ne n$, then
\begin{equation*}
\left|\int_{J_T}e^{-i\tau\log(m/n)}\,d\tau\right|
\le\frac{2}{|\log(m/n)|}
\le\frac{2N}{|m-n|}.
\end{equation*}
Using $2|d_m d_n|\le|d_m|^2+|d_n|^2$ and
\begin{equation*}
\sum_{\substack{n\le N\\n\ne m}}\frac1{|m-n|}
\ll\log(2N),
\end{equation*}
bounds the off-diagonal terms by the error in
\eqref{eq:weighted-mean-value}.  This proves the lemma.
\end{proof}

Lemmas~\ref{lem:coefficient-truncation} and
\ref{lem:weighted-mean-value} give the following estimates.  Here
$N=T^{1/2}$, so $N\log(2N)=o(T)$; in particular, for all sufficiently
large $T$, the off-diagonal error is at most half of the diagonal term.
\begin{align}
\mathcal M_\gamma(S_N)
&\gg \eta^{-A_2}T B(T,\mathbf b,\mathbf a),
                                                        \label{eq:S-gamma-lower}\\
\mathcal M_{1/2}(S_N)
&\ll T B(T,\mathbf b,\mathbf a).                \label{eq:S-half-upper}
\end{align}

\section{Convexity in a strip}\label{sec:convexity}

We use the following form of Gabriel's theorem.  If $V$ is analytic in the
strip $\alpha\le\Re s\le\beta$, tends to zero uniformly as
$|\Im s|\to\infty$, and
\begin{equation*}
I_V(\sigma)=\int_{-\infty}^{\infty}|V(\sigma+it)|^p\,dt
\end{equation*}
is finite on the boundary, then, for $p>0$ and
$\alpha\le\sigma\le\beta$,
\begin{equation}\label{eq:gabriel-original}
I_V(\sigma)
\le I_V(\alpha)^{(\beta-\sigma)/(\beta-\alpha)}
I_V(\beta)^{(\sigma-\alpha)/(\beta-\alpha)}.
\end{equation}
The inequality is due to Gabriel \cite{Gabriel}.  We use Proposition 3.3 of
Arguin, Ouimet, and Radziwi\l\l{} \cite{ArguinOuimetRadziwill}, which records
the required strip formulation.

\begin{lem}\label{lem:weighted-gabriel}
Let $U$ be analytic in a neighbourhood of the closed strip
$\alpha\le\Re s\le\beta$, and suppose that, for some $C\ge0$,
\begin{equation*}
|U(\sigma+it)|\ll(1+|t|)^C
\end{equation*}
uniformly for $\alpha\le\sigma\le\beta$.  Suppose its boundary values have
finite weighted norms.  If
$\alpha\le\sigma\le\beta$ and
\begin{equation*}
\vartheta=\frac{\sigma-\alpha}{\beta-\alpha},
\end{equation*}
then
\begin{equation}\label{eq:weighted-gabriel}
\mathcal M_\sigma(U)
\le \exp\!\left(\vartheta(1-\vartheta)(\beta-\alpha)^2\right)
\mathcal M_\alpha(U)^{1-\vartheta}
\mathcal M_\beta(U)^{\vartheta}.
\end{equation}
More generally, suppose that $U$ is analytic in the open strip, has
polynomial growth uniformly on its closed substrips, and that
\begin{equation*}
\mathcal M_{\alpha+\delta}(U)\longrightarrow\mathcal M_\alpha(U),
\qquad
\mathcal M_{\beta-\delta}(U)\longrightarrow\mathcal M_\beta(U)
\end{equation*}
as $\delta\downarrow0$, with finite limits.  Then the same conclusion
holds with these one-sided endpoint norms.
\end{lem}

\begin{proof}
Fix $\tau\in J_T$ and apply \eqref{eq:gabriel-original}, with $p=2$, to
\begin{equation*}
V_\tau(s)=U(s)\exp((s-i\tau)^2/2).
\end{equation*}
The Gaussian factor supplies the required decay.  Since
\begin{equation*}
|V_\tau(\sigma+it)|^2
=e^{\sigma^2-(t-\tau)^2}|U(\sigma+it)|^2,
\end{equation*}
equation~\eqref{eq:gabriel-original} gives
\begin{align*}
\int_{\R}|U(\sigma+it)|^2e^{-(t-\tau)^2}\,dt
&\le e^{(1-\vartheta)\alpha^2+\vartheta\beta^2-\sigma^2}\\
&\quad\times
\left(\int_{\R}|U(\alpha+it)|^2e^{-(t-\tau)^2}\,dt\right)^{1-\vartheta}\\
&\quad\times
\left(\int_{\R}|U(\beta+it)|^2e^{-(t-\tau)^2}\,dt\right)^{\vartheta}.
\end{align*}
Because $\sigma=(1-\vartheta)\alpha+\vartheta\beta$, the exponent in
front is $\vartheta(1-\vartheta)(\beta-\alpha)^2$.
Integrating over $\tau\in J_T$ and applying H\"older's inequality proves
\eqref{eq:weighted-gabriel}.  The one-sided version follows by first using
$\alpha+\delta$ and $\beta-\delta$.  The interpolation parameters tend to
$\vartheta$, the exponential factors remain bounded by
$\exp((\beta-\alpha)^2/4)$, and the endpoint norms converge by hypothesis.
Letting $\delta\downarrow0$ gives the asserted limiting inequality.
\end{proof}

We now remove the shifted poles.  By Lemma~\ref{lem:regularised-zeta}, the
half-plane $\Re z>1/2$ admits an analytic logarithm of $\Phi_{K_i}(z)$.
On $\Re z>1$, let
\begin{equation*}
\ell(z)=\operatorname{Log}(z-1)-\operatorname{Log}(z+1),
\end{equation*}
where both logarithms are real on the positive real axis.  This is the
unique analytic logarithm of $(z-1)/(z+1)$ for which
$\ell(x)\to0$ as $x\to+\infty$.  We normalize the logarithm of
$\Phi_{K_i}$ by
\begin{equation*}
\log\Phi_{K_i}(z)=\log_{\mathrm{Euler}}\zeta_{K_i}(z)+\ell(z)
\qquad(\Re z>1),
\end{equation*}
and continue it to $\Re z>1/2$.  Define
\begin{equation}\label{eq:F-tilde-definition}
\tilde F_{\mathbf b}(s)=
\prod_{i=1}^r\Phi_{K_i}(s+ib_i)^{q_i}
\qquad(\Re s>1/2).
\end{equation}
This function is analytic even when the $q_i$ are not integers.  In
$\Re s>1$, the branches in \eqref{eq:F-dirichlet-series} and
\eqref{eq:F-tilde-definition} satisfy
\begin{equation}\label{eq:regularisation-identity}
\tilde F_{\mathbf b}(s)=F_{\mathbf b}(s)
\prod_{i=1}^r
\left(\frac{s+ib_i-1}{s+ib_i+1}\right)^{q_i}.
\end{equation}
With $N=T^{1/2}$, let
\begin{equation*}
R_N(s)=\tilde F_{\mathbf b}(s)-S_N(s).
\end{equation*}

We record the boundary meaning of \eqref{eq:F-tilde-definition}.  Suppose
$\rho$ is a zero of $\Phi_{K_i}$ on $\Re z=1/2$, of multiplicity $m$.
In a neighbourhood of $\rho$,
\begin{equation*}
\Phi_{K_i}(z)=(z-\rho)^m g(z),\qquad g(\rho)\ne0.
\end{equation*}
On the intersection of this neighbourhood with $\Re z>1/2$, the chosen
fractional power differs by a constant of modulus one from compatible
branches of $(z-\rho)^{m q_i}g(z)^{q_i}$.  It therefore tends to zero when
$z$ tends to $\rho$ from the right.  Away from the
boundary zeros, $\Phi_{K_i}$ has a local analytic logarithm, and the chosen
global logarithm differs from it by an element of $2\pi i\mathbb Z$.
It consequently has a finite one-sided limit.  Thus
$\tilde F_{\mathbf b}$, and hence $R_N$, has the boundary values used
below; at a boundary zero we use the convention $0^{q_i}=0$.

We also record a majorant for these boundary values.

\begin{lem}\label{lem:strip-growth}
There is a constant
$C\ge1$, depending only on the fixed fields and $\mathbf a$, such that,
uniformly for $0<\delta\le1/4$, $1/2+\delta\le\sigma\le3/2$,
$|b_i|\le T/2$, and $t\in\R$,
\begin{equation}\label{eq:boundary-majorant}
|\tilde F_{\mathbf b}(\sigma+it)|^2
+|S_N(\sigma+it)|^2
+|R_N(\sigma+it)|^2
\ll (1+T+|t|)^C.
\end{equation}
\end{lem}

\begin{proof}
The functional equation, Stirling's formula, and the
Phragm\'en--Lindel\"of principle give polynomial growth for each fixed
Dedekind zeta function in a fixed vertical strip.  The apparent pole of
$\Phi_{K_i}$ at $1$ is removable, so the same conclusion holds for
$\Phi_{K_i}$ throughout $1/2\le\Re z\le3/2$.  Since the exponents $q_i$
are positive and real,
$|\Phi_{K_i}(z)^{q_i}|=|\Phi_{K_i}(z)|^{q_i}$, independently of the
chosen logarithm.  The inequality $|t+b_i|\le |t|+T/2$ therefore gives
the required bound for $\tilde F_{\mathbf b}$.  Lemma
\ref{lem:coefficient-bounds} and $N=T^{1/2}$ give a fixed polynomial bound
for $S_N$, and the estimate for $R_N$ follows from the triangle
inequality.
\end{proof}

From \eqref{eq:gaussian-weight} we also have
\begin{equation}\label{eq:gaussian-majorant}
w_T(t)\le \frac{T}{2}
\exp\left(-\operatorname{dist}(t,J_T)^2\right).
\end{equation}
Consequently the right side of \eqref{eq:boundary-majorant}, multiplied by
$w_T(t)$, is integrable and is independent of $\delta$, $\sigma$, and
$\mathbf b$ in the stated ranges.

\begin{lem}\label{lem:right-error}
Let $N=T^{1/2}$ and let $R_N$ be as above.
Uniformly for $|b_i|\le T/2$,
\begin{align}
\mathcal M_{5/4}(R_N)&\ll T^{7/8},                    \label{eq:right-error}\\
\mathcal M_{3/2}(\tilde F_{\mathbf b})&\ll T.     \label{eq:right-F}
\end{align}
\end{lem}

\begin{proof}
Lemma~\ref{lem:coefficient-bounds}, with $\epsilon=1/16$, gives uniformly
in $t$ and $\mathbf b$
\begin{equation}\label{eq:dirichlet-tail}
|F_{\mathbf b}(5/4+it)-S_N(5/4+it)|
\ll\sum_{n>N}n^{-5/4+1/16}
\ll N^{-3/16}=T^{-3/32}.
\end{equation}
If $T\le t\le2T$ and $|b_i|\le T/2$, then
$|t+b_i|\ge T/2$.  The normalization above gives
\begin{equation*}
\ell(5/4+i(t+b_i))
=\operatorname{Log}\left(1-\frac{2}{9/4+i(t+b_i)}\right)
=O(T^{-1}).
\end{equation*}
Equation~\eqref{eq:regularisation-identity} and absolute convergence on
$\Re s=5/4$ therefore show that
\begin{equation}\label{eq:regularisation-error}
|\tilde F_{\mathbf b}(5/4+it)-F_{\mathbf b}(5/4+it)|
\ll T^{-1}.
\end{equation}
On the whole line $\Re s=5/4$, absolute convergence and Lemma
\ref{lem:coefficient-bounds} show that $F_{\mathbf b}$ and $S_N$ are
uniformly bounded.  Moreover,
\begin{equation*}
\left|\frac{s+ib_i-1}{s+ib_i+1}\right|\le1
\qquad(\Re s=5/4),
\end{equation*}
so \eqref{eq:regularisation-identity} gives the same bound for
$\tilde F_{\mathbf b}$, and hence for $R_N$.  For $t\notin[T,2T]$ and
$\tau\in J_T$, one has $|t-\tau|\ge T/4$.  Integrating first in $t$ and
then in $\tau$ gives
\begin{equation*}
\int_{\R\setminus[T,2T]}w_T(t)\,dt\ll e^{-T^2/20}.
\end{equation*}
Since
$\int_{\R}w_T(t)\,dt=\sqrt\pi T/2$, equations
\eqref{eq:dirichlet-tail} and \eqref{eq:regularisation-error} yield
\begin{equation*}
\mathcal M_{5/4}(R_N)
\ll T\bigl(T^{-3/32}+T^{-1}\bigr)^2+O(e^{-T^2/20})
\ll T^{13/16}
\ll T^{7/8}.
\end{equation*}
On $\Re s=3/2$, absolute convergence bounds the zeta factors, while the
rational factors in \eqref{eq:regularisation-identity} have modulus at
most one.  Integration of the weight gives \eqref{eq:right-F}.
\end{proof}

\begin{prop}\label{prop:weighted-lower}
For all sufficiently large $T$, uniformly for $|b_i|\le T/2$,
\begin{equation}\label{eq:weighted-lower}
\mathcal M_{1/2}(\tilde F_{\mathbf b})
\gg T B(T,\mathbf b,\mathbf a).
\end{equation}
Here the norm at $1/2$ is the one-sided limit from $\Re s>1/2$.
\end{prop}

\begin{proof}
Fix $\eta\ge\eta_0$, to be chosen below, and let $\gamma$ be as in
\eqref{eq:N-gamma}.  Put $B=B(T,\mathbf b,\mathbf a)$.  The preceding
estimates furnish constants $c_0,C_0,C_1>0$, depending only on the fixed
fields and $\mathbf a$, such that
\begin{align*}
\mathcal M_\gamma(S_N)&\ge c_0\eta^{-A_2}TB,
&\mathcal M_{1/2}(S_N)&\le C_0TB,\\
\mathcal M_{5/4}(R_N)&\le C_1T^{7/8},
&\mathcal M_{3/2}(\tilde F_{\mathbf b})&\le C_1T.
\end{align*}
If $\mathcal M_{1/2}(\tilde F_{\mathbf b})\ge TB$, the proposition already
holds.  We may therefore assume
\begin{equation}\label{eq:small-boundary-alternative}
\mathcal M_{1/2}(\tilde F_{\mathbf b})\le TB.
\end{equation}
Equations \eqref{eq:S-half-upper} and
\eqref{eq:small-boundary-alternative}, together with the $L^2$ triangle
inequality, give
\begin{equation}\label{eq:R-half-upper}
\mathcal M_{1/2}(R_N)\ll TB.
\end{equation}

Apply Lemma~\ref{lem:weighted-gabriel} to $R_N$ between $1/2$ and $5/4$.
The interpolation parameter at $\gamma$ is
\begin{equation*}
\vartheta_1=\frac{4\eta}{3\log T}.
\end{equation*}
By \eqref{eq:right-error}, \eqref{eq:R-half-upper}, and
\eqref{eq:B-poly},
\begin{align}
\mathcal M_\gamma(R_N)
&\ll (TB)^{1-\vartheta_1}(T^{7/8})^{\vartheta_1}\notag\\
&\ll TB\exp\left(-\frac{\eta}{6}
+O\left(\frac{\eta\log\log T}{\log T}\right)\right).  \label{eq:R-gamma-upper}
\end{align}
Thus there is a fixed $C_2>0$ such that, for every fixed
$\eta\ge\eta_0$ and all sufficiently large $T$,
\begin{equation*}
\mathcal M_\gamma(R_N)\le2C_2e^{-\eta/6}TB.
\end{equation*}
Choose $\eta$ so large that
\begin{equation}\label{eq:eta-choice}
2C_2e^{-\eta/6}\le \frac{c_0}{4}\eta^{-A_2}.
\end{equation}
It follows that
$\mathcal M_\gamma(R_N)\le\mathcal M_\gamma(S_N)/4$.  Hence the triangle
inequality in weighted $L^2$ gives
\begin{align*}
\mathcal M_\gamma(\tilde F_{\mathbf b})^{1/2}
&\ge \mathcal M_\gamma(S_N)^{1/2}
-\mathcal M_\gamma(R_N)^{1/2}\\
&\ge\tfrac12\mathcal M_\gamma(S_N)^{1/2},
\end{align*}
and therefore
\begin{equation}\label{eq:F-gamma-lower}
\mathcal M_\gamma(\tilde F_{\mathbf b})
\ge \frac{c_0}{4}\eta^{-A_2}TB\gg_\eta TB.
\end{equation}

Apply Lemma~\ref{lem:weighted-gabriel} once more, now to
$\tilde F_{\mathbf b}$ between $1/2$ and $3/2$.  The interpolation
parameter is $\vartheta_2=\eta/\log T$.  Equations
\eqref{eq:F-gamma-lower} and \eqref{eq:right-F} give
\begin{equation*}
TB\ll_\eta
\mathcal M_{1/2}(\tilde F_{\mathbf b})^{1-\vartheta_2}
T^{\vartheta_2}.
\end{equation*}
Write
$X=\mathcal M_{1/2}(\tilde F_{\mathbf b})/(TB)$.  By
\eqref{eq:B-poly}, uniformly in the shifts,
\begin{equation*}
B^{-\vartheta_2}
=\exp\!\left(O\left(\frac{\eta\log\log T}{\log T}\right)\right)
\le2
\end{equation*}
for sufficiently large $T$.  After division by $TB$, the preceding
estimate gives
\begin{equation*}
c_\eta\le X^{1-\vartheta_2}B^{-\vartheta_2}
\le2X^{1-\vartheta_2}
\end{equation*}
for a fixed $c_\eta>0$.  Since $\vartheta_2\le1/2$ once $T$ is large,
this implies $X\ge\min(1,(c_\eta/2)^2)$.  This proves
\eqref{eq:weighted-lower}.

To justify the boundary passages, apply Lemma~\ref{lem:weighted-gabriel}
first on the strips $[1/2+\delta,5/4]$ and
$[1/2+\delta,3/2]$, keeping the right endpoints fixed.  Equations
\eqref{eq:boundary-majorant} and
\eqref{eq:gaussian-majorant} give a polynomial majorant times a Gaussian,
uniformly for $|b_i|\le T/2$ and $0<\delta\le1/4$.  Dominated convergence
at each fixed $T$ and $\mathbf b$ therefore permits $\delta\downarrow0$
in every boundary norm used above; the dominating function and Gabriel
constants are uniform in $\mathbf b$.  At a zero on the critical line the
modulus of the corresponding fractional power tends to zero.  No analytic
continuation of that fractional power across the line is required.
\end{proof}

\section{Proof of the main theorem}\label{sec:main-proof}

\begin{proof}[Proof of Theorem~\ref{thm:main}]
For $T\le t\le2T$ and $|b_i|\le T/2$, one has $|t+b_i|\ge T/2$ and hence
\begin{equation*}
\left|\frac{1/2+i(t+b_i)-1}{1/2+i(t+b_i)+1}\right|\asymp1.
\end{equation*}
It follows from \eqref{eq:Phi-definition} that
\begin{equation}\label{eq:regularised-integrand-comparison}
|\tilde F_{\mathbf b}(1/2+it)|^2
\asymp
\prod_{i=1}^r
\left|\zeta_{K_i}\left(\tfrac12+i(t+b_i)\right)\right|^{a_i}
\end{equation}
throughout this interval.  Also, $w_T(t)\le\sqrt\pi$.  Taking the one-sided
limit in \eqref{eq:boundary-majorant} gives the same polynomial bound for
$|\tilde F_{\mathbf b}(1/2+it)|^2$.  If $t\notin[T,2T]$, put
$d=\operatorname{dist}(t,J_T)$; then $d\ge T/4$.  After increasing $C$ by
one to absorb the factor $T/2$ in
\eqref{eq:gaussian-majorant},
\begin{equation*}
(1+T+|t|)^C e^{-d^2}\ll_C e^{-4d^2/5}
\end{equation*}
for sufficiently large $T$.  Integration therefore gives
\begin{equation}\label{eq:weight-tail}
\int_{\R\setminus[T,2T]}
|\tilde F_{\mathbf b}(1/2+it)|^2w_T(t)\,dt
\ll e^{-T^2/20},
\end{equation}
uniformly for $|b_i|\le T/2$.  Proposition~\ref{prop:weighted-lower},
\eqref{eq:regularised-integrand-comparison}, and \eqref{eq:weight-tail}
therefore give
\begin{equation*}
TB(T,\mathbf b,\mathbf a)
\ll
\int_T^{2T}\prod_{i=1}^r
\left|\zeta_{K_i}\left(\tfrac12+i(t+b_i)\right)\right|^{a_i}\,dt
+O(e^{-T^2/20}).
\end{equation*}
The lower estimate in \eqref{eq:B-poly} absorbs the last term.  This proves
Theorem~\ref{thm:main}.
\end{proof}

\section{The unshifted consequences}\label{sec:corollaries}

\begin{proof}[Proof of Corollary \ref{cor:unshifted-mixed}]
Every Dedekind zeta function in \eqref{eq:double-coset-zeta} has a simple
pole at $1$ with positive residue.  Thus
\begin{equation*}
\calZ_{ij}\left(1+\frac1{\log T}\right)
\asymp (\log T)^{|H_i\backslash G/H_j|}.
\end{equation*}
The diagonal action of $G$ on $G/H_i\times G/H_j$ has
$|H_i\backslash G/H_j|$ orbits and character $\chi_i\chi_j$.
Burnside's lemma therefore gives
\begin{equation}\label{eq:burnside-double-coset}
|H_i\backslash G/H_j|
=\langle\chi_i,\chi_j\rangle_G.
\end{equation}
Set every shift equal to zero in Theorem~\ref{thm:main}, and use
\eqref{eq:burnside-double-coset}.  This proves the corollary.
\end{proof}

\begin{proof}[Proof of Corollary \ref{cor:single-field}]
Take $r=1$ and $a_1=2k$ in Corollary
\ref{cor:unshifted-mixed}.  Then
$\langle\chi,\chi\rangle_G=|H\backslash G/H|=\rho_K$, which proves
\eqref{eq:single-field}.  For a non-Galois cubic with Galois closure group
$S_3$, the subgroup $H$ is a point stabiliser.  Its action on the three
points has two orbits: the fixed point and the other two points.  Hence
$|H\backslash S_3/H|=2$.  For representatives of the two double cosets,
the corresponding subgroups $H\cap gHg^{-1}$ are conjugate to $H$ and
$\{1\}$, respectively.  Thus the double-coset definition also gives
\begin{equation*}
\calZ_{11}(s)=\zeta_K(s)\zeta_L(s).
\end{equation*}
\end{proof}

\begin{proof}[Proof of Corollary \ref{cor:galois}]
If $K/\Q$ is Galois, then $H$ is normal in $G$.  The double cosets are the
cosets of $H$, so
\begin{equation*}
\rho_K=|G/H|=[K:\Q].
\end{equation*}
Corollary \ref{cor:single-field} now gives \eqref{eq:galois}.
\end{proof}

\section*{Acknowledgements}
BD is supported by the Heilbronn Institute for Mathematical Research.
NK is supported by the Prime Minister's Research Fellowship (PMRF),
Government of India (PMRF ID:~2403449).  KM is supported by the
ARG-MATRICS programme (Grant No.~ANRF/ARGM/2025/002540/MTR).

\enlargethispage{4\baselineskip}

\end{document}